\documentclass[10pt,reqno]{amsart}
\usepackage{amsmath,amssymb,verbatim}
\usepackage[utf8]{inputenc}
\pdfoutput=1
\usepackage{enumerate}
\usepackage{pdflscape}
\usepackage{amsthm}
\usepackage{mathrsfs}
\usepackage{color}
\usepackage{multicol}
\usepackage[normalem]{ulem}
\usepackage{cancel}
\usepackage{tikz}
\usetikzlibrary{matrix}
\usepackage[all]{xy}
\usepackage{mathtools}
\usepackage{arydshln}
\usepackage[shortlabels]{enumitem}
  \usepackage[english]{babel}
\usepackage{hyperref}
\hypersetup{
    colorlinks = true,
    linkbordercolor = {white},
    allcolors = blue
} 
\usepackage[capitalize, noabbrev]{cleveref}
 
\makeatletter
\@namedef{subjclassname@1991}{\textup{2020} Mathematics Subject Classification}
\makeatother

\topskip=\baselineskip 
\newtheorem{theorem}{Theorem}[section]

\newtheorem{fact}[theorem]{Fact}

\newtheorem{example}[theorem]{Example}

\newtheorem*{theorem*}{Theorem}

\newtheorem{maintheorem}{Theorem}

\newtheorem{maincorollary}[maintheorem]{Corollary}

\makeatletter
\@addtoreset{equation}{section}
\makeatother

\newcommand{\Z}{{\mathbb Z}}
\newcommand{\Q}{{\mathbb Q}}

\newcommand{\F}{{\mathbb F}}

\newcommand{\p}{\mathbb P}
\newcommand{\Min}{\operatorname{Min}}

\newcommand{\qand}{\quad \text{and} \quad}

\title[A reduction theorem for the Isomorphism Problem of group algebras over fields]{A reduction theorem for the Isomorphism Problem  \\  of group algebras over fields}

\author{Diego Garc\'{\i}a-Lucas  }
 \address{Departamento de Matem\'aticas, Universidad de Murcia, Spain}
 \email{diego.garcial@um.es}

\author{Ángel del R\'{\i}o}
  \address{Departamento de Matem\'aticas, Universidad de Murcia, Spain}
 \email{adelrio@um.es}

 \thanks{Partially supported by Grant PID2020-113206GB-I00 funded by MCIN/AEI/10.13039/501100011033  and by Grant Fundación Séneca 22004/PI/22. }



\date{\today}

\begin{document}

\begin{abstract}
We prove that the Isomorphism Problem for group algebras reduces to group algebras over finite extensions of the prime field. In particular, the Modular Isomorphism Problem reduces to finite modular group algebras.
\end{abstract}

\maketitle

Let $R$ be commutative ring. The Isomorphism Problem for group rings over $R$ asks whether the isomorphism type of $G$ is determined by the isomorphism type of $RG$ as $R$-algebra.
More precisely:

\begin{quote}
\textbf{Isomorphism Problem for Group Rings}: Let $R$ be a commutative ring and let $G$ and $H$ be groups. If $RG$ and $RH$ are isomorphic as $R$-algebras, are $G$ and $H$ isomorphic groups?
\end{quote}

This problem has been studied extensively, with special emphasis on the cases where the groups are finite and the coefficient ring is either the integers or a field. One of the first results is due to G. Higman who proved that if $G$ and $H$ are finite abelian groups and $\Z G\cong \Z H$, then $G\cong H$ \cite{HigmanThesis,Higman40}.
 This was extended to finite metabelian groups by  A. Whitcomb   \cite{Whitcomb},   and  later extended  to  abelian-by-nilpotent groups by K. W. Roggenkamp and L. Scott \cite{RoggenkampScott1987} and to nilpotent-by-abelian groups by W. Kimmerle (see \cite[Section~XII]{RoggenkampTaylor}).
However, M. Hertweck founded two non-isomorphic finite solvable groups with isomorphic integral group rings \cite{Hertweck2001}.
See \cite[Section~3]{MargolisdelRioSurvey} for an overview on the Isomorphism Problem for integral group rings.

In this  note we only consider finite groups and fields as coefficient rings.
It is easy to obtain non-isomorphic finite groups with isomorphic group algebras. For example, if $G$ and $H$ are abelian groups of the same order  $n$  and $F$ is an algebraically closed field of characteristic not dividing $n$, then $FG\cong FH$.
However, by a result of S. Perlis and G. L. Walker, if $G$ and $H$ are non-isomorphic abelian groups, then $\Q G\not\cong \Q H$ \cite{PW50}.   A more elaborated counterexample, given by E. Dade, consists in two non-isomorphic finite metabelian groups $G$ and $H$ with $FG\cong FH$ for every field $F$ \cite{Dade71}. The order of these  groups 
is divisible by two different primes.
On the other hand, D. S. Passman proved the following two results (see \cite{Passman65} or \cite[Theorems~1.9 and 1.11]{Passman77}). If $G$ and $H$ are finite groups with $\Q G\cong \Q H$, then $FG\cong FH$ for every field $F$ of characteristic coprime with the order of $G$. If $p$ is prime and $n$ is a positive integer, then there are at least $p^{\frac{2}{27}n^2(n-23)}$ pairwise non-isomorphic groups of order $p^n$ with isomorphic group algebras over any field of characteristic different from $p$.

These motivated the so-called Modular Isomorphism Problem (abbreviated MIP), which is the version of the Isomorphism Problem for group algebras of finite $p$-groups over the field with $p$ elements, or more generally over  fields of characteristic $p>0$.
 This turned out   to be quite different, and definitive answers more difficult to obtain. Some partial positive solutions where obtained by a number of authors  \cite{Passman1965p4,PS72, BC88,San89,Drensky,SalimSandlingp5,Bag99,Her07,EK11,BK19,BdR20, MM20,MSS21,MS22,GLdRS2022}, and recently a negative solution for $p=2$ was given in \cite{GarciaMargolisdelRio}.

 The previous discussion shows that, even in the same characteristic, changing the field may alter  the answer to the Isomorphism Problem.
 Yet, if the Isomorphism Problem for $FG$ and $FH$ has a positive solution and $K$ is a subfield of $F$, then it also has a positive solution for $KG$ and $KH$ because $FG\cong F\otimes_K KG$.
Hence, the larger the field,  the greater the chances for the Isomorphism Problem to have a negative answer are. The aim of this note is to bound the class of coefficient fields when searching for negative solutions for the Isomorphism Problem. Namely, we prove the following:

\begin{maintheorem}	\label{theorem:reductionFields}
Let $F$ be a field, $\p$ the prime field of $F$, and $G$ and $H$ finite groups. If $F G\cong  F H$, then there exist a finite extension $F_0$ of $\p$ such that $F_0G\cong F_0H$.
\end{maintheorem}

If the characteristic of $\p$ is coprime with the order of the group $G$, then
a finite extension of $\p$ split $\p G$ and $\p H$ and hence, in this case, a proof of \Cref{theorem:reductionFields} is straightforward`. However, we present a unified proof for any characteristic.
In contrast with \Cref{theorem:reductionFields}, it is easy to see that there is no absolute bound for the index of $\p$ in the field $F_0$ of \Cref{theorem:reductionFields} (see \Cref{proposition}).

The application of \Cref{theorem:reductionFields} to the MIP shows that this question can be regarded as  exclusively about finite objects. Formally:
\begin{maincorollary}
	\label{cor:reductionFields}
	Let $G$ and $H$ finite $p$-groups such that $F G\cong  F H$ for some field $F$ of characteristic $p$. Then there exists a finite field $F_0$ of characteristic $p$ such that $F_0G\cong F_0H$.
\end{maincorollary}

The fact that this last reduction is not (to the best of our knowledge)   mentioned elsewhere is somehow surprising, since, although most of the results known about the MIP depend heavily on the primality of the field (e.g. \cite{San89}), some authors were interested in the question substituting the prime field for arbitrary fields  of characteristic $p$. For instance, in \cite{Drensky}, V. Drensky showed that the Isomorphism Problem for fields of characteristic $p$ and finite $p$-group with center of index $p^2$ has a positive solution , even though the result for the prime field was already known by a result of I. B. N. Passi and S. K. Sehgal \cite{PS72}.
Further classical results on the MIP which are only stated over the prime field are known to hold for arbitrary fields of the same characteristic (see e.g. \cite{San96}).
A complete list of results about the MIP, distinguishing which ones are known for every field of characteristic $p$, and which ones only for the prime field, can be found in the recent survey \cite{Mar22}.
Moreover,  some of the techniques used to study the MIP, mainly the ones consisting on counting elements in the group algebra verifying some property, as suggested by R. Brauer in \cite{Bra63}, are only available for finite fields.
This, apparently, makes  the reduction in \Cref{cor:reductionFields} a potentially useful tool in the study of the extended version of the MIP.


The state of the art in the extended version of the MIP suggests the following question.

\medskip
\noindent\textbf{Question}.
	Let $G$ and $H$ be finite $p$-groups such that $FG\cong FH$ for some field $F$ of characteristic $p$. Does it imply that $\F_p G\cong \F_pH$, for the field $\F_p$ with $p$ elements?

\bigskip

For the proofs we use standard algebraic notation. For example, given rings $R\subseteq S$ and $T$ a subset of $S$, $R[T]$ denotes the subring of $S$ generated by $R\cup T$.
In case $R$ and $S$ are fields, $R(T)$ denotes the smallest subfield of $S$ containing $R\cup T$.
When $T=\{t_1,\dots,t_n\}$ we write $R[t_1,\dots,t_n]$ and $R(t_1,\dots,t_n)$ rather than $R[\{t_1,\dots,t_n\}]$ or $R(\{t_1,\dots,t_n\})$.
If $\psi:R\to R'$ is a ring homomorphism and $Z$ is an indeterminate, then we also denote $\psi$ the natural extension of $\psi$ to a homomorphism $R[Z]\to R'[Z]$.
Let $E/F$ be a field extension. If $\alpha\in E$ is algebraic over $F$, then $\Min_F(\alpha)$ denotes the minimal polynomial of $\alpha$ over $F$.
Observe that if $FG\cong FH$ for groups $G$ and $H$, then $EG\cong EH$, because $EG\cong E\otimes_F FG$. We will use this frequently to replace $F$ by a convenient overfield.

We will use another trivial observation:

 \begin{fact}\label{fact3.1}
Let $F$ be a field and $\phi: F G\to  F H$ be an algebra isomorphism.
Write $\phi(g)=\sum_{h\in H}a_{gh} h$ for each $g\in G$, where $a_{gh}\in  F$.
If $L=\p (a_{gh}: g\in G, h\in H)$, then $LG\cong LH$.
\end{fact}

\begin{proof}
Immediate, as the homomorphism $\tilde \phi :LG\to LH$ given by  $\tilde \phi(g)=\sum_{h\in H} a_{gh} h $ is an isomorphism, because the matrix $(a_{gh})_{g\in G,h\in H}$ is non-singular.
\end{proof}

\begin{proof}[Proof of \Cref{theorem:reductionFields}]
Let $F$ be a field, and let $G$ and $H$ be finite groups such that $F G\cong  F H$.
Let $\bar  F$ be an algebraic closure of $F$.
All our fields will be subfields of $\bar F$.
By \Cref{fact3.1} we can assume that $F=\p(a_{gh}: g\in G,h\in H)$.
In particular $F$ is finitely generated over $\p$.
By \cite[\S 14 Theorem 1]{BourbakiAlgebraII}, there is an intermediate extension $\p\subseteq E\subseteq  F$ with $E/\p$ purely transcendental,
$F/E$ algebraic and both finitely generated.
Hence $E=\p(X_1,\dots, X_m)$, where the $X_i$'s are algebraically independent over $\p$ and replacing $F$ by the normal closure of $F/E$ we may assume that $F/E$ is a finite normal extension.

\begin{fact}\label{fact2}
There is a finite algebraic extension $L$ of $E$ such that $\p\subseteq L$ is purely transcendental with finite transcendence degree and $LF/L$ is a finite Galois extension.
\end{fact}

\begin{proof}
Let $p$ be the characteristic of $F$. The statement is clear if $p=0$ so suppose that $p>0$.
By \cite[Theorem 19.18]{IsaacsAlgebra} there is a subextension $E \subseteq K\subseteq  F$ with $K/E$ purely inseparable, and $F/K$ separable. Observe that all these extensions are finite and $F/K$ is Galois.

Let $\alpha_1,\dots, \alpha_m\in K$ such that $K=E(\alpha_1,\dots, \alpha_{ m_i})$. Each $\alpha_i$ is the unique root of a polynomial $T^{p^{m_i}}-\beta_i$ for some $\beta_i\in E$.
Let $M=\max\{m_i: i=1,\dots, m\}$ and let $T_i$ be the unique root of $T^{p^M}-X_i$ in $\bar  F$, for $i\in\{1,\dots, m\}$.
Set $L=E(T_1,\dots,T_m)=\p(T_1,\dots, T_m)$.
Then $L/\p$ is purely transcendental with finite transcendence degree and, as $F/K$ is finite and Galois, to prove that so is $LF/L$, it suffices to show that $K\subseteq L$.
As $E=\p(X_1,\dots,X_m)$, each $\beta_i=\frac{\delta_i(X_1,\dots,X_m)}{R_i(X_1,\dots,X_m)}$ for some $\delta_i,R_i\in \p[X_1,\dots,X_m]$ and therefore $\tilde \beta_i=\frac{\delta_i(T_1,\dots,T_m)}{R_i(T_1,\dots,T_m)}$ satisfies $\tilde \beta_i^{p^M}=\beta_i$.
Then $\alpha_i$ is the unique root of $T^{p^{m_i}}-\beta_i=(T-\tilde \beta_i^{p^{M-m_i}})^{p^{m_i}}$, so $\alpha_i=\tilde \beta_i^{p^{M-m_i}}\in L$.
Hence $K\subseteq L$, as desired.
\end{proof}

Let $L$ be a field as in \Cref{fact2}. By replacing $E$ and $F$ by $L$ and $LF$ respectively, we may assume that $F/E$ is a finite Galois extension.
Then, by the Primitive Element Theorem, $F=E(\zeta)$ for some $\zeta\in  F$.

Let $R=\p[X_1,\dots, X_m]\subseteq E$.
Let $d$ be the determinant of the matrix $(a_{gh})_{g\in G, h\in H}$, $A=\Min_{E}(\zeta)$ and $B=\Min_{E}(d)$.
Then
	$$a_{gh}=\frac{\sum_{i=0}^{n_{gh}} P_{ghi} \zeta^i}{Q_{gh}},\quad
	A=\sum_i \frac{\alpha_i}{\beta_i} Z^i \in E[Z] \qand
	B=\sum_i \frac{\gamma_i}{\delta_i} Z^i \in E[Z],$$
for some $P_{ghi},Q_{gh },\alpha_i, \beta_i,\gamma_i,\delta_i\in R$, with each $Q_{gh},\beta_i$ and $\delta_i$ different from $0$.
Let
	$$Q=\gamma_0  \prod_i \delta_i \prod_i \beta_i\prod_{\tiny \begin{array}{c}
		g\in G, \\
		h\in H 
	\end{array}} Q_{gh} \in \p[X_1,\dots, X_n].$$
Let $S=\{Q^k : k\ge 0\}$, a multiplicative subset of $R$, and let $D=S^{-1}R$ be the ring of fractions of $R$ by $S$.
As $R$ is a unique factorization domain, so is $D$.
Let $\bar \p$ be the algebraic closure of $\p$ in $\bar{F}$.
By Hilbert's Nullstellensatz, there is an evaluation homomorphism $\psi:R\to \bar \p$ such that $\psi(Q)\neq 0$.
By the universal property of the ring of fractions, $\psi$ extends to a homomorphism $D\to \bar \p$, which we also denote $\psi$.
As $A\in D[Z]$, by Gauss Lemma the kernel of the evaluation homomorphism $\delta:D[Z]\to D[\zeta]$ mapping $Z$ to $\zeta$ is generated by $A$.
Let $\bar \zeta$ be a root of the polynomial $\psi(A)$ in $\bar \p$.
The homomorphism $\psi':D [Z]\to \bar \p$ which extends $\psi$ and maps  $Z$  to $\bar \zeta$ has $A$ in its kernel, so there is a homomorphism
$\psi'':D[\zeta] \to \bar \p$ making commutative the diagram $$\xymatrix{
	D[Z] \ar[dr]_-{\psi'} \ar[r]^-{\delta } & D[\zeta] \ar[d]^-{\psi''} \\
	& \bar \p
}. $$
Let $\bar a_{gh }=\psi''(a_{gh})$ and $\bar d=\psi''(d)$.
Then $\bar d$ is the determinant of the matrix $(\bar a_{gh})_{g\in G,h\in H}$, and a root of the polynomial $\psi'' (B)$. As the independent term of
$\psi''(B)$ is $\psi''(P_0)=\psi(P_0)\neq 0$, we conclude that $\bar d\neq 0$. Let $F_0=\p(\bar a_{gh}: g\in G, h\in H)\subseteq \bar \p$. Then $F_0$ is a finite extension of $\p$. Moreover, the map $G\mapsto (F_0H)^\times$ given by $g\mapsto \sum_{h\in H}  \bar a_{gh} h$ is an homomorphism of groups, as it is just the composition of the inclusion $G\subseteq D[\zeta] G$, the restriction $\phi|_{D[\zeta]G}:D[\zeta]G\to D[\zeta]H$, and the homomorphism $D[\zeta] H\to F_0H$ induced by $\psi'':D[\zeta]\to F_0$. Therefore there is a homomorphism of $F_0$-algebras $\bar \phi: F_0G\to F_0H$ such that $\bar \phi(g)=\sum_{h\in H}   \bar a_{gh} h$ for each $g\in G$. As $\bar d\neq 0$, the set $\bar \phi(G)$ is a basis of $F_0H$. Thus $\bar \phi$ is an isomorphism.
\end{proof}


We denote the cyclic group of order $n$ by $C_n$.
For $m$ and $n$ coprime integers, let $o_m(n)$ denote the multiplicative order of $n$ modulo $m$, i.e. the smallest positive integer $t$ with $n^t\equiv 1 \mod m$.

\begin{example}\label{proposition}
For every prime field $\p$ and every positive integer $n$ there exist a prime integer $q$ different than the characteristic of $\p$, and a positive integer $k$ such that $FC_q^k\not\cong FC_{q^k}$ for every field extension $F$ of $\p$ with $[F:\p]\le n$, but $EC_q^k\cong EC_{q^k}$ for some field extension $E$ of $\p$.
\end{example}

\begin{proof}
We first suppose that $\p$ is of characteristic $0$. In this case we take $q=2$ and $k$ the least positive integer with $2^{k-1}>n$.
To prove that $q$ and $k$ satisfy the desired condition we use a Theorem of Perlis and Walker \cite{PW50} (see also \cite[Theorem~3.3.6]{JespersdelRioGRG1}). Let $F$ be a field extension of $\p$. Then $FC_2^k\cong F^{2^k}$, and if $F$ contains a primitive $2^k$-root of unity $\zeta$, then $FC_{2^k}\cong F^{2^k}$. Thus, if $E=\p(\zeta)$ then $EC_2^k\cong EC_{2^k}$, while if $[F:\p]\le n$ then $\zeta\not\in F$, so $FC_{2^k}$ has a factor isomorphic to $F(\zeta)$ and hence $FC_2^k\not\cong FC_{2^k}$.

For positive characteristic we use a version of the Perlis-Walker Theorem for finite fields (see \cite[Problem~3.3.9]{JespersdelRioGRG1}).

Suppose that $\p$ has characteristic $p>2$ and let $q$ be a prime divisor of $p-1$. The sequence $(o_{q^k}(p))_k$ is unbounded and we take a positive integer $k$ such that $o_{q^k}(p)>n$.
Let $F$ be a finite field extension of $\p$ with $t=[F:\p]$.
As in the previous case $FC_q^k\cong F^{q^k}$ and if $p^t\equiv 1 \mod q^k$, then $FC_{q^k}\cong F^{q^k}$.
However, if $t\le n$, then $FC_{q^k}$ has an epimorphic image which is a field extension of $F$ of degree $o_{q^k(p^t)}=\frac{o_{q^k}(p)}{\gcd(t,o_{q^k}(p))}>1$ because $t<o_{q^k}(p)$. Thus $FC_q^k\not\cong FC_{q^k}$.

Finally, suppose that $\p$ has characteristic $2$. In this case we take $q=3$ and $k$ the least integer such that $3^{k-1}>n$.
Let $F$ be a field extension of $\p$ of degree $t$ and $K$ a field extension of $F$ of degree $2$. If $t$ is even, then $FC_3=F^3$. However, if $t$ is odd, then $FC_3\cong F\times K$. Therefore, $FC_3^k$ is a direct product of fields isomorphic to $F$ or $K$, and if $t$ is even, then $FC_3^k\cong F^{3^k}$.
If $2^t\equiv 1 \mod 3^k$, then $t$ is even and $FC_{3^k}\cong F^{3^k}$.
However, if $t\le n$, then $FC_{3^k}$ has an epimorphic image isomorphic to a field extension of $F$ of degree $o_{3^k}(2^t)=\frac{o_{3^k}(2)}{\gcd(t,o_{3^k}(2))}=\frac{2\cdot 3^{k-1}}{\gcd(t,2\cdot 3^{k-1})}>2$ because $t<3^{k-1}$. Thus $FC_3^k\not\cong FC_{3^k}$.
\end{proof}

 
 \bibliographystyle{amsalpha}
 \bibliography{MIP} 

\providecommand{\bysame}{\leavevmode\hbox to3em{\hrulefill}\thinspace}
\providecommand{\MR}{\relax\ifhmode\unskip\space\fi MR }
\providecommand{\MRhref}[2]{%
  \href{http://www.ams.org/mathscinet-getitem?mr=#1}{#2}
}
\providecommand{\href}[2]{#2}
\begin{thebibliography}{GLMdR22}

\bibitem[Bag99]{Bag99}
C.~Bagi\'{n}ski, \emph{On the isomorphism problem for modular group algebras of
  elementary abelian-by-cyclic {$p$}-groups}, Colloq. Math. \textbf{82} (1999),
  no.~1, 125--136.

\bibitem[BC88]{BC88}
C.~Bagi\'{n}ski and A.~Caranti, \emph{The modular group algebras of
  {$p$}-groups of maximal class}, Canad. J. Math. \textbf{40} (1988), no.~6,
  1422--1435.

\bibitem[BdR21]{BdR20}
O.~Broche and \'{A}. del R\'{\i}o, \emph{The {M}odular {I}somorphism {P}roblem
  for two generated groups of class two}, Indian J. Pure Appl. Math.
  \textbf{52} (2021), 721--728.

\bibitem[BK19]{BK19}
C.~Bagi\'{n}ski and J.~Kurdics, \emph{The modular group algebras of
  {$p$}-groups of maximal class {II}}, Comm. Algebra \textbf{47} (2019), no.~2,
  761--771.

\bibitem[Bou90]{BourbakiAlgebraII}
N.~Bourbaki, \emph{Algebra {II}: Chapters 4-7 ({P}t.2)}, Springer, 1990.

\bibitem[Bra63]{Bra63}
R.~Brauer, \emph{Representations of finite groups}, Lectures on {M}odern
  {M}athematics, {V}ol. {I}, Wiley, New York, 1963, pp.~133--175.

\bibitem[Dad71]{Dade71}
E.~Dade, \emph{Deux groupes finis distincts ayant la m\^{e}me alg\`ebre de
  groupe sur tout corps}, Math. Z. \textbf{119} (1971), 345--348.

\bibitem[Dre89]{Drensky}
V.~Drensky, \emph{The isomorphism problem for modular group algebras of groups
  with large centres}, Representation theory, group rings, and coding theory,
  Contemp. Math., vol.~93, Amer. Math. Soc., Providence, RI, 1989,
  pp.~145--153.

\bibitem[EK11]{EK11}
B.~Eick and A.~Konovalov, \emph{The modular isomorphism problem for the groups
  of order 512}, Groups {S}t {A}ndrews 2009 in {B}ath. {V}olume 2, London Math.
  Soc. Lecture Note Ser., vol. 388, Cambridge Univ. Press, Cambridge, 2011,
  pp.~375--383.

\bibitem[GLdRS22]{GLdRS2022}
D.~García-Lucas, Á. del Río, and M.~Stanojkovski, \emph{On group invariants
  determined by modular group algebras: even versus odd characteristic},
  https://arxiv.org/abs/2209.06143.

\bibitem[GLMdR22]{GarciaMargolisdelRio}
D.~García-Lucas, L.~Margolis, and Á. del Río, \emph{Non-isomorphic
  $2$-groups with isomorphic modular group algebras}, J. Reine Angew. Math.
  \textbf{154} (2022), no.~783, 269--274.

\bibitem[Her01]{Hertweck2001}
M.~Hertweck, \emph{A counterexample to the isomorphism problem for integral
  group rings}, Ann. of Math. (2) \textbf{154} (2001), no.~1, 115--138.

\bibitem[Her07]{Her07}
Martin Hertweck, \emph{A note on the modular group algebras of odd {$p$}-groups
  of {$M$}-length three}, Publ. Math. Debrecen \textbf{71} (2007), no.~1-2,
  83--93.

\bibitem[Hig40a]{HigmanThesis}
G.~Higman, \emph{Units in group rings}, 1940, Thesis (Ph.D.)--Univ. Oxford.

\bibitem[Hig40b]{Higman40}
\bysame, \emph{The units of group-rings}, Proc. London Math. Soc. (2)
  \textbf{46} (1940), 231--248.

\bibitem[Isa94]{IsaacsAlgebra}
I.~Martin Isaacs, \emph{Algebra, a graduate course}, 1 ed., Mathematics,
  Brooks/Cole Pub. Co, 1994.

\bibitem[JdR16]{JespersdelRioGRG1}
E.~Jespers and {\'A}.~del R{\'{\i}}o, \emph{{Group ring groups. Volume 1:
  Orders and generic constructions of units}}, Berlin: De Gruyter, 2016.

\bibitem[{Mar}22]{Mar22}
L.~{Margolis}, \emph{{The Modular Isomorphism Problem: A Survey}}, Jahresber.
  Dtsch. Math. Ver. (2022).

\bibitem[MM20]{MM20}
L.~{Margolis} and T.~{Moede}, \emph{{The Modular Isomorphism Problem for small
  groups -- revisiting Eick's algorithm}}, arXiv:2010.07030,
  https://arxiv.org/abs/2010.07030.

\bibitem[MR18]{MargolisdelRioSurvey}
L.~Margolis and {\'A}.~del R{\'i}o, \emph{Finite subgroups of group rings: A
  survey}, preprint, arxiv.org/abs/1809.00718 (2018), 23 pages.

\bibitem[MS22]{MS22}
L.~Margolis and M.~Stanojkovski, \emph{On the modular isomorphism problem for
  groups of class $3$ and obelisks}, J. Group Theory \textbf{25} (2022), no.~1,
  163--206.

\bibitem[MSS21]{MSS21}
L.~Margolis, T.~Sakurai, and M.~Stanojkovski, \emph{Abelian invariants and a
  reduction theorem for the modular isomorphism problem},
  https://arxiv.org/abs/2110.10025.

\bibitem[Pas65a]{Passman1965p4}
D.~S. Passman, \emph{The group algebras of groups of order {$p^{4}$} over a
  modular field}, Michigan Math. J. \textbf{12} (1965), 405--415. \MR{0185022}

\bibitem[Pas65b]{Passman65}
\bysame, \emph{Isomorphic groups and group rings}, Pacific J. Math. \textbf{15}
  (1965), 561--583.

\bibitem[Pas77]{Passman77}
\bysame, \emph{The algebraic structure of group rings}, Pure and Applied
  Mathematics, Wiley-Interscience [John Wiley \& Sons], New York-London-Sydney,
  1977.

\bibitem[PS72]{PS72}
I.~B.~S. Passi and S.~K. Sehgal, \emph{Isomorphism of modular group algebras},
  Math. Z. \textbf{129} (1972), 65--73.

\bibitem[PW50]{PW50}
S.~Perlis and G.~L. Walker, \emph{Abelian group algebras of finite order},
  Trans. Amer. Math. Soc. \textbf{68} (1950), 420--426.

\bibitem[RS87]{RoggenkampScott1987}
K.~W. Roggenkamp and L.~Scott, \emph{Isomorphisms of {$p$}-adic group rings},
  Ann. of Math. (2) \textbf{126} (1987), no.~3, 593--647.

\bibitem[RT92]{RoggenkampTaylor}
K.~W. Roggenkamp and M.~J. Taylor, \emph{Group rings and class groups}, DMV
  Seminar, vol.~18, Birkh\"auser Verlag, Basel, 1992.

\bibitem[San89]{San89}
R.~Sandling, \emph{The modular group algebra of a central-elementary-by-abelian
  {$p$}-group}, Arch. Math. (Basel) \textbf{52} (1989), no.~1, 22--27.

\bibitem[San96]{San96}
\bysame, \emph{The modular group algebra problem for metacyclic {$p$}-groups},
  Proc. Amer. Math. Soc. \textbf{124} (1996), no.~5, 1347--1350.

\bibitem[SS96]{SalimSandlingp5}
M.~A.~M. Salim and R.~Sandling, \emph{The modular group algebra problem for
  groups of order {$p^5$}}, J. Austral. Math. Soc. Ser. A \textbf{61} (1996),
  no.~2, 229--237.

\bibitem[Whi68]{Whitcomb}
A.~Whitcomb, \emph{The {G}roup {R}ing {P}roblem}, ProQuest LLC, Ann Arbor, MI,
  1968, Thesis (Ph.D.) -- The University of Chicago.

\end{thebibliography}

\end{document}